\documentclass[a4paper, reqno]{amsart} 
\usepackage[margin=3.5cm]{geometry}

\usepackage{amsthm, amssymb, amsmath, latexsym}
\usepackage{amsfonts, mathrsfs, color}

\theoremstyle{plain}                     
\begingroup
\newtheorem{theorem}{Theorem}[section]

\endgroup
      
\theoremstyle{definition}                
\begingroup
\newtheorem{definition}[theorem]{Definition}

\endgroup

\newcommand{\e}{\varepsilon}

\newcommand{\R}{\mathbb R}
\newcommand{\N}{\mathbb N}
\newcommand{\HH}{\mathcal H}
\newcommand{\LL}{\mathcal L}

\newcommand{\mres}{\mathbin{\vrule height 1.6ex depth 0pt width
0.13ex\vrule height 0.13ex depth 0pt width 1.3ex}}

\newcommand{\Div}{\mathrm{div}\,}
\newcommand{\dist}{\mathrm{dist}}
\newcommand{\supp}{\mathrm{supp}}
\newcommand{\sym}{\mathrm{sym}}
\newcommand{\Mn}{\mathbb M^{n\times n}_{sym}}
\newcommand{\MD}{\mathbb M^{n\times n}_D}
\newcommand{\Mnn}{\mathbb M^{n\times n}}

\newcommand{\wto}{\rightharpoonup}

\numberwithin{equation}{section}

\title[Relaxation of the Hencky model in perfect plasticity]{
Relaxation of the Hencky model 
in perfect plasticity} 

\author[M.G. Mora]{Maria Giovanna Mora}

\address[M.G. Mora]{Dipartimento di Matematica, Universit\`a di Pavia, via Ferrata 1, 27100 Pavia (Italy)}
\email{mariagiovanna.mora@unipv.it}

\subjclass[2010]{49J45 (74C05, 74G65)} 
\keywords{Perfect plasticity, Hencky model, functions with bounded deformation, relaxation.}
 
\begin{document}

\begin{abstract}
In this paper we give a full proof of the relaxation of the Hencky model in perfect plasticity, under suitable assumptions
for the domain and the Dirichlet boundary.
\end{abstract}

\maketitle

\section{Introduction}
The first complete mathematical treatment of the evolution problem in perfect plasticity
is due to Suquet \cite{S}. More recently, in \cite{DMDSM} plastic evolution has been revisited 
in the framework of the variational theory for rate-independent processes (see, e.g., \cite{MM,M}).
In this variational approach existence of a quasi-static evolution is proved by approximation via time discretization and by solving a suitable incremental minimum problem at each discrete time. 


More precisely, let $\Omega\subset\R^n$ denote the reference configuration of a plastic body.
The elasto-plastic behaviour of $\Omega$ is described by three kinematic variables: the displacement $u:\Omega\to\R^n$,
the elastic strain $e:\Omega\to\Mn$, and the plastic strain $p:\Omega\to\MD$. Here $\MD$ is the subspace of trace-free matrices in
$\Mn$. Moreover, the strain $Eu:=\sym\, Du$ is related to $e$ and $p$ by the following kinematic admissibility condition:
$$
Eu=e+p \quad \text{ in } \Omega.
$$
The requirement $p(x)\in\MD$ for every $x\in\Omega$ corresponds to the plastic incompressibility condition ${\rm tr}\, p=0$ in $\Omega$,
which is a usual requirement in the description of plastic behaviour in metals.

Let now $[0,T]$ be a time interval and assume for simplicity that the evolution is driven by a time-dependent boundary displacement $w:[0,T]\times\R^n\to\R^n$ 
prescribed on a portion $\Gamma_0$ of $\partial\Omega$.
Let $t_i$ be a given discrete time and let $(u_{i-1},e_{i-1},p_{i-1})$ be the elasto-plastic configuration of the body at
the previous discrete time $t_{i-1}$. Then,
the configuration $(u_i,e_i,p_i)$ at time $t_i$ is found as a solution of the minimum problem
\begin{multline}\label{min pb}
\min \Big\{ 
\int_\Omega Q(e(x))\, dx +\int_\Omega H(p(x)-p_{i-1}(x))\, dx : \\ 
(u,e,p) \text{ such that } Eu=e+p \text{ in }\Omega, \ u=w(t_i) \text{ on }\Gamma_0\Big\}.
\end{multline}
Here $Q:\Mn\to[0,+\infty)$ is a positive definite quadratic form, representing the elastic energy density.
The function $H:\MD\to[0,+\infty)$ plays the role of a plastic dissipation potential and it is defined as the support function 
of a given convex and closed set $K\subset\MD$, that is,
\begin{equation}\label{defH}
H(\xi):=\sup_{\sigma\in K} \sigma:\xi \quad \text{ for every } \xi\in\MD.
\end{equation}
The set $K$ represents the elasticity domain and its boundary $\partial K$ defines the so-called yield surface. For $p_{i-1}=0$  problem \eqref{min pb} is usually referred to as the Hencky model of perfect plasticity.

From the definition \eqref{defH} it easily follows that the function $H$ is convex and has linear growth.
Thus, the natural domain of the functional in \eqref{min pb} is the class $\mathcal A_{reg}(w(t_i), \Gamma_0)$ of all triplets $(u,e,p)\in LD(\Omega)\times L^2(\Omega;\Mn)\times L^1(\Omega;\MD)$ such that
$$
Eu=e+p\quad \text{in }\Omega, \qquad
u=w(t_i)\quad\text{on }\Gamma_0.
$$
We recall that $LD(\Omega)$ is the space of $L^1(\Omega;\R^n)$ functions whose symmetric gradient is in $L^1(\Omega;\Mn)$.

However, since $LD(\Omega)$ and $L^1(\Omega;\MD)$ are not reflexive spaces,
problem \eqref{min pb} has in general no solution in the class $\mathcal A_{reg}(w(t_i), \Gamma_0)$.
For this reason, in \cite{DMDSM}, as well as in the subsequent literature, problem \eqref{min pb} is replaced by the following
weak formulation:
\begin{equation}\label{min pb 2}
\min \Big\{ 
\int_\Omega Q(e(x))\, dx +\HH_{\Omega\cup\Gamma_0}(p-p_{i-1}) : \ (u,e,p)\in \mathcal A(w(t_i), \Gamma_0)\Big\},
\end{equation}
where $\HH_{\Omega\cup\Gamma_0}(p-p_{i-1})$ is defined according to the theory of convex functions of measures (see Section~\ref{sec:math-p} for more details) and the class
$\mathcal A(w(t_i), \Gamma_0)$ is the set of all triplets $(u,e,p)\in BD(\Omega)\times L^2(\Omega;\Mn)\times M_b(\Omega\cup\Gamma_0;\MD)$ such that
$$
Eu=e+p\quad \text{in }\Omega, \qquad
p=(w(t_i)-u)\odot\nu_{\partial\Omega}\HH^{n-1}\quad\text{on }\Gamma_0.
$$
In other words, the plastic strain $p$ is allowed to take values in a space of measures (and thus, the displacement $u$
in the space $BD(\Omega)$ of functions with bounded deformation) and the boundary condition is relaxed (the boundary value may be not attained by $u$ and in this case a discontinuity is developed along $\Gamma_0$).

It is easy to see that problem \eqref{min pb 2} is an extension of problem \eqref{min pb}, meaning that any solution to \eqref{min pb} solves \eqref{min pb 2}, as well. In the isotropic case with von Mises yield criterion it was shown in
\cite[Theorem~2.3]{AG} and in \cite[Chapter~II, Section~6.2]{T} that \eqref{min pb} and \eqref{min pb 2} have the same infimum.
In this article we prove that the relation between the two problems is much stronger: \eqref{min pb 2} is the relaxed problem of \eqref{min pb} in the sense of $\Gamma$-convergence, with respect to the natural topology, and as such, it is its natural extension (Theorem~\ref{thm:relax}). In particular, by the abstract theory of $\Gamma$-convergence \cite{DM} this implies that not only the two problems have the same infimum, but also the minimisers of \eqref{min pb 2} coincide with all the limits of minimising sequences of \eqref{min pb}.

\

A fundamental step in establishing Theorem~\ref{thm:relax} is to prove the density of the class $\mathcal A_{reg}(w(t_i), \Gamma_0)$ in the class $\mathcal A(w(t_i), \Gamma_0)$. Density is intended with respect to a topology that guarantees convergence of the energies, that is, strong-$L^2$ convergence for the elastic strains and strict convergence in the sense of measures for the plastic strains. This question is highly non trivial. Indeed, by the kinematic admissibility and the plastic incompressibility condition, displacements in $\mathcal A(w(t_i), \Gamma_0)$ belong to the space
$$
U(\Omega):=\{ u\in BD(\Omega): \ \Div u\in L^2(\Omega)\}.
$$
For $n>2$ this space is not local: if $u\in U(\Omega)$, then it is not true in general that $\varphi u\in U(\Omega)$
for every $\varphi\in C^\infty_c(\Omega)$. This fact prevents any na\"\i ve approximation based on local arguments and partitions of unity.
Moreover, since displacements in $\mathcal A_{reg}(w(t_i), \Gamma_0)$ attain exactly the boundary condition on $\Gamma_0$, one needs to correct the boundary value, again without leaving the space $U(\Omega)$, before any regularization procedure.

In Section~\ref{sec:density} we prove two versions of this density result, under two different sets of assumptions for the domain $\Omega$ and the Dirichlet boundary $\Gamma_0$. In Theorem~\ref{thm:density} we assume $\partial\Omega$ to be of class $C^{2,1}$ and $\Gamma_0$ to be any open subset of $\partial\Omega$, while in Theorem~\ref{thm:density-dir} we consider the full Dirichlet case $\Gamma_0=\partial\Omega$ with $\partial\Omega$ of Lipschitz regularity. 
We believe these density results to be of independent interest. For instance, Theorem~\ref{thm:density-dir} has been applied in the recent paper \cite{C} to characterise the asymptotic behaviour of a certain family of quasistatic evolutions in a strain gradient plasticity model coupled with damage.

A crucial ingredient in the proof of both Theorems~\ref{thm:density} and~\ref{thm:density-dir} is a result by Bogovsky \cite[Theorem~1]{B} (see also \cite[Theorem~2.4]{B-S}), stating the following: there exists a constant $C>0$ such that for every function $\psi\in L^p(\Omega)$ with null average on $\Omega$, there exists a solution $v\in W^{1,n/(n-1)}_0(\Omega;\R^n)$ 
to the problem   
$$
\begin{cases}
\Div v=\psi & \text{ in }\Omega,\\
v=0 & \text{ on }\partial\Omega.
\end{cases}
$$
satisfying the estimate
$$
\|v\|_{W^{1,n/(n-1)}}\leq C \|\psi\|_{L^{n/(n-1)}}.
$$
By local mollifications we first construct smooth approximating triplets $(u_k,e_k,p_k)$ with $\Div u_k\in L^{n/(n-1)}(\Omega)$.
Bogovski's result allows us to correct the displacements $u_k$ in such a way to gain the $L^2$-integrability of the divergence.

Concerning the boundary condition issue, in the Dirichlet case (Theorem~\ref{thm:density-dir}) we extend $u$ by 
the boundary datum $w$ outside $\Omega$ and, before mollifying, we perform a local translation of the boundary towards the interior of $\Omega$. If $\Gamma_0\neq\partial\Omega$ (Theorem~\ref{thm:density}), we apply a clever construction by Anzellotti and Giaquinta \cite{AG}, that requires higher regularity of the boundary. In Theorem~\ref{AGlemma}
we provide a full proof of this construction, since the original proof in \cite{AG} contains several misprints and inaccuracies.

We finally mention that an approximation result close in spirit to Theorems~\ref{thm:density} and~\ref{thm:density-dir} has been proved in \cite[Theorem~4.7]{DaM} for a model of perfectly plastic plates, but the proof is more conventional and does not require Bogovski's argument, since the model under consideration can be partially reduced to a two-dimensional setting.

\

This paper provides an answer to a basic question in the mathematical theory of perfect plasticity: what is the exact relation between the classical formulation \eqref{min pb} of the Hencky model and
the weak formulation \eqref{min pb 2}? Although this is a very natural question, it has gone unregarded by the mathematical
community and has been left open, up to the present contribution. The only established results available in the literature were those
 in \cite{AG} and in \cite{T}, concerning the isotropic case with von Mises yield condition and showing only equality of infima. However, we are much indebted to these works, since several arguments in the proofs of this article were inspired by ideas contained in \cite{AG} and in~\cite{T}.

\section{Mathematical Preliminaries}\label{sec:math-p}

In this section we introduce some notation and recall some notions that will be used throughout the article.
\medskip

\noindent{\bf Distance function.}
Let $U$ be a bounded open set of $\R^n$ with a $C^2$ boundary and let $d(x):=\dist(x,\partial U)$ for every $x\in \overline U$.
It is well known that there exists $a>0$ such that, setting $U_a:=\{ x\in\overline U:\ d(x)<a\}$, then for every $x\in U_a$ there exists a unique projection
$\pi(x)\in \partial U$ such that $d(x)=|x-\pi(x)|$; moreover, $d$ is differentiable in $U_a$ and $\nabla d(x)=-\nu_{\partial U}(\pi(x))$ for every $x\in U_a$,
where $\nu_{\partial U}$ is the outward unit normal vector to $\partial U$. If $\partial U\in C^k$ with $k\geq 2$, then $d\in C^k(U_a)$. See, e.g., \cite[Section~14.6]{GT}.\medskip

\noindent{\bf Matrices.} 
The space of symmetric $n\times n$ matrices is denoted by $\Mn$. It is endowed
with the euclidean scalar product $\xi{\,:\,}\zeta= \sum_{i,i}\xi_{ij}\zeta_{ij}$ and the
corresponding euclidean norm $|\xi|^2=(\xi{\,:\,}\xi)^{1/2}$. The orthogonal complement of the
subspace $\R I_{n\times n}$ spanned by the identity matrix $I_{n\times n}$ 
is the subspace $\MD$ of all matrices in $\Mn$ with trace zero. 
For every $\xi\in\Mn$ the orthogonal projection of $\xi$ on $\MD$ is the deviator $\xi_D$ of $\xi$, given by
$$
\xi_D=\xi-\frac1n ({\rm tr}\,\xi)I_{n\times n}.
$$

The symmetrised tensor product $a\odot b$ of two vectors $a,b\in\R^n$ is the 
symmetric matrix with entries $(a\odot b)_{ij}=(a_ib_j + a_jb_i)/2$.\medskip

\noindent{\bf Measures.}
Given a Borel set $B\subset\R^n$ and a finite dimensional Hilbert space $X$, $M_b(B;X)$ denotes the space of all bounded
Borel measures on $B$ with values in $X$, endowed with the norm $\|\mu\|_{M_b}:=|\mu|(B)$, where
$|\mu|\in M_b(B;\R)$ is the variation of the measure $\mu$. If $\mu$ is absolutely continuous with respect to the
Lebesgue measure $\LL^n$, we always identify $\mu$ with its density with respect to $\LL^n$, which is a function in $L^1(B;X)$.

If the relative topology of $B$ is locally compact, by the Riesz representation Theorem the space $M_b(B;X)$ can be identified with the dual of $C_0(B;X)$, which is the space of all continuous functions $\varphi:B\to X$ such that
the set $\{|\varphi|\geq\delta\}$ is compact for every $\delta>0$. The weak* topology on $M_b(B;X)$ is defined
using this duality. Finally, we say that a sequence of measures $\mu_k$ converges to $\mu$ strictly in $M_b(B;X)$
if $\mu_k\wto\mu$ weakly$^*$ in $M_b(B;X)$ and $|\mu_k|(B)\to |\mu|(B)$.
\medskip

\noindent{\bf Convex functions of measures.}
Let $U$ be an open set of $\R^n$. 
For every $\mu\in M_b(U;X)$ let $d\mu/d|\mu|$ be the Radon-Nikod\'ym derivative of $\mu$ with respect to
its variation $|\mu|$. Let $H:X\to[0,+\infty)$ be a convex and positively one-homogeneous function such that
$$
r|\xi|\leq H(\xi)\leq R|\xi| \quad \text{for every }\xi\in X,
$$
where $r$ and $R$ are two constants, with $0<r\leq R$. According to the theory of convex functions
of measures, developed in \cite{GS}, we introduce the nonnegative Radon measure
$H(\mu)\in M_b(U)$ defined by
$$
H(\mu)(A):=\int_A H\Big(\frac{d\mu}{d|\mu|}\Big)\, d|\mu|
$$
for every Borel set $A\subset U$. We also consider the functional $\HH_U: M_b(U;X)\to[0,+\infty)$ defined~by
$$
\HH_U(\mu):=H(\mu)(U)=\int_U H\Big(\frac{d\mu}{d|\mu|}\Big)\, d|\mu|
$$
for every $\mu\in M_b(U;X)$. One can prove that $\HH_U$ 
is lower semicontinuous on $M_b(U;X)$ with respect to weak* convergence (see, e.g., \cite[Theorem~2.38]{AFP}).
\medskip

\noindent{\bf Functions with bounded deformation.}
Let $U$ be an open set of $\R^n$. 
The space $BD(U)$ of functions with {\em bounded deformation} is
the space of all functions $u\in L^1(U;\R^n)$ whose symmetric gradient $Eu:=\sym\, Du$ (in the sense of distributions) belongs to $M_b(U;\Mn)$. It is easy to see that $BD(U)$ is a 
Banach space endowed with the norm
$$
\|u\|_{BD}:=\|u\|_{L^1} +\|Eu\|_{M_b}.
$$
We say that a sequence $(u^k)$ converges to $u$ weakly* in $BD(U)$ if $u^k\wto u$ weakly in
$L^1(U;\R^n)$ and $Eu^k\wto Eu$ weakly* in $M_b(U;\Mn)$. Every bounded
sequence in $BD(U)$ has a weakly* converging subsequence. If $U$ is bounded and has a Lipschitz boundary, $BD(U)$ can be embedded into $L^{n/(n-1)}(U;\R^n)$ and every function $u\in BD(U)$ has a trace,
still denoted by $u$, which belongs to $L^1(\partial U;\R^n)$. Moreover, if $\Gamma$ is a nonempty open subset of $\partial U$, there exists a constant $C>0$, depending on $U$ and $\Gamma$, such that
\begin{equation}\label{kornbd}
\|u\|_{L^1(\Omega)}\leq C\|u\|_{L^1(\Gamma)}+C\|Eu\|_{M_b}
\end{equation}
(see \cite[Chapter~II, Proposition~2.4 and Remark~2.5]{T}). 

We will also use the space $LD(U)$ defined as the space of all functions $u\in L^1(U;\R^n)$ whose symmetric gradient $Eu:=\sym\, Du$ belongs to $L^1(U;\Mn)$. The space $LD(U)$ is a 
Banach space endowed with the norm
$$
\|u\|_{LD}:=\|u\|_{L^1} +\|Eu\|_{L^1}.
$$
For the general properties of the spaces $BD(U)$ and $LD(U)$ we refer to~\cite{T}.

\section{Two Density Results}\label{sec:density}

In this section we prove two density results in the class of admissible triplets. 
We first introduce some notation.

\begin{definition}
Let $w\in W^{1,2}(\R^n;\R^n)$ and let $\Gamma_0$ be an open subset of $\partial\Omega$ (in the relative topology). The class $\mathcal A_{reg}(w,\Gamma_0)$ of regular triplets with boundary datum $w$ is defined as the set of all triplets $(u,e,p)\in LD(\Omega)\times L^2(\Omega;\Mn)\times L^1(\Omega;\MD)$ such that
\begin{align}
& Eu=e+p\quad \text{a.e.\ in }\Omega, \label{Areg1}
\\
& u=w \quad\text{on }\Gamma_0. \label{Areg2}
\end{align}
The class $\mathcal A(w, \Gamma_0)$ of triplets with boundary datum $w$ is defined as the set of all triplets $(u,e,p)\in BD(\Omega)\times L^2(\Omega;\Mn)\times M_b(\Omega\cup\Gamma_0;\MD)$ such that
\begin{align}
& Eu=e+p\quad \text{in }\Omega, \label{A1}
\\
& p=(w-u)\odot\nu_{\partial\Omega}\HH^{n-1}\quad\text{on }\Gamma_0. \label{A2}
\end{align}
\end{definition}

The first result of this section is an approximation result for triplets in $\mathcal A(w, \Gamma_0)$ in terms of regular triplets in
$\mathcal A_{reg}(w,\Gamma_0)$.

\begin{theorem}\label{thm:density}
Let $\Omega\subset \R^n$ be an open and bounded set with a $C^{2,1}$ boundary.
Let $w\in W^{1,2}(\R^n;\R^n)$ and let $(u,e,p)\in\mathcal A(w, \Gamma_0)$. Then there exists a sequence of triplets $(u_k, e_k,p_k)$ in $\mathcal A_{reg}(w,\Gamma_0)$ such that
\begin{align}
& u_k\to u \quad \text{strongly in } L^{n/(n-1)}(\Omega;\R^n), \label{uk}\\
& e_k\to e \quad \text{strongly in } L^2(\Omega;\Mn), \label{ek}\\
& p_k\wto p \quad \text{weakly$^*$ in } M_b(\Omega\cup\Gamma_0;\MD), \label{pk1} 
\end{align}
and
\begin{equation}\label{pk2}
\int_\Omega |p_k(x)|\,dx \to |p|(\Omega\cup\Gamma_0),
\end{equation}
as $k\to\infty$.
\end{theorem}

The proof of Theorem~\ref{thm:density} is based on the following auxiliary result, that was stated in~\cite{AG}.
We give here a complete proof of this result, since the proof proposed in \cite{AG} contains several inaccuracies and misprints.

\begin{theorem}\label{AGlemma}
Let $\Omega\subset \R^n$ be an open and bounded set with a $C^{2,1}$ boundary.
Let $d(x):=\dist(x,\partial\Omega)$ for every $x\in\overline\Omega$ and let $\Omega_a:=\{x\in\overline\Omega: d(x)<a\}$ be such that
$d\in C^2(\Omega_a)$.
Finally, let $u\in L^1(\partial\Omega;\R^n)$ be such that $u\cdot \nu_{\partial\Omega}=0$ on $\partial\Omega$.
Then there exists $v\in W^{1,1}(\Omega;\R^n)\cap L^2(\Omega;\R^n)$, with $\Div v\in L^2(\Omega)$, such that $\supp\, v\subset \Omega_a$, 
$v=u$ on $\partial\Omega$, and
$$
v(x)\cdot\nabla d(x)=0 
$$ 
for a.e.\ $x\in\Omega_a$.
\end{theorem}

\begin{proof}
We introduce the following notation:
$$
Q:=(-1,1)^n, \qquad Q^+:=(-1,1)^{n-1}{\times}(0,1), \qquad Q_0:=(-1,1)^{n-1}{\times}\{0\}.
$$
The proof is subdivided into two steps.\medskip

\noindent{\it Step 1.}
We first prove that, given $u\in L^1((-1,1)^{n-1};\R^{n-1})$, there exists a function $v\in W^{1,1}(Q^+;\R^n)\cap L^2(Q^+;\R^n)$, with $\partial_i v\in L^2(Q^+;\R^n)$ for $i=1,\dots,n-1$, such that $v=(u,0)$ on $Q_0$ and $ v\cdot e_n=0$ a.e.\ in $Q^+$.

Let $u\in L^1((-1,1)^{n-1};\R^{n-1})$ be given.
Let $\{\tau_j\}$ be a sequence of positive numbers decreasing to $0$, as $j\to\infty$, and 
let $\{\theta_j\}\subset C^\infty_c((-1,1)^{n-1};\R^{n-1})$ be such that $\theta_0\equiv0$ and 
\begin{equation}\label{trace-approx}
\theta_j\to u \quad \text{strongly in } L^1((-1,1)^{n-1};\R^{n-1}), 
\end{equation}
as $j\to\infty$. We denote the coordinates in $Q^+$ by $(x', x_n)\in (-1,1)^{n-1}{\times}(0,1)$
and we define
$$
v(x',x_n):=\begin{cases}
0 & \text{ for } x_n\geq \tau_0,
\smallskip \\
\Big(\theta_j(x') +\dfrac{x_n-\tau_j}{\tau_{j+1}-\tau_j}\big(\theta_{j+1}(x')-\theta_j(x')\big)  ,0\Big) & \text{ for } \tau_{j+1}\leq x_n < \tau_j,\ j\geq0.
\end{cases}
$$
It is clear that $v\cdot e_n=0$ in $Q^+$.
By straightforward computations we have that
$$
\|v\|^2_{L^2}\leq \sum_{j=0}^\infty (\tau_j-\tau_{j+1})\big(\|\theta_j\|_{L^2}^2+\|\theta_{j+1}\|_{L^2}^2\big),
$$
and analogously,
$$
\|\partial_i v\|^2_{L^2}\leq \sum_{j=0}^\infty (\tau_j-\tau_{j+1})\big(\|\partial_i\theta_j\|_{L^2}^2+\|\partial_i\theta_{j+1}\|_{L^2}^2\big),
$$
for $i=1,\dots,n-1$, while
$$
\|\partial_n v\|_{L^1}= \sum_{j=0}^\infty \|\theta_j -\theta_{j+1}\|_{L^1}.
$$
Therefore, a suitable choice of the convergence rates of $\{\tau_j\}$ and $\{\theta_j\}$ guarantees that $v\in W^{1,1}(Q^+;\R^n)\cap L^2(Q^+;\R^n)$, with $\partial_i v\in L^2(Q^+;\R^n)$ for $i=1,\dots,n-1$. Finally, in view of \eqref{trace-approx},
the trace of $v$ on $Q_0$ coincides with $(u,0)$.

We also note that, if $\supp\, u\subset \omega$ where $\omega$ is an open set compactly contained in $(-1,1)^{n-1}$, then we can choose the sequence 
$\{\theta_j\}$ in such a way that $\supp\,\theta_j\subset\omega$ for every $j$, so that $\supp\, v\subset \omega\times[0,1)$.\medskip

\noindent{\it Step 2.} We now prove the general statement. Let $u\in L^1(\partial\Omega;\R^n)$ be such that $u\cdot \nu_{\partial\Omega}=0$ on $\partial\Omega$. We can cover $\partial\Omega$ with a finite number of open sets $A_j$, $j=1,\dots, N$, such that
$\cup_j A_j\cap\Omega\subset\Omega_a$ and 
for every $j=1,\dots, N$ there exists a $C^{1,1}$ diffeomorphism $\Phi_j:A_j\to Q$ satisfying $\Phi_j(A_j\cap\Omega)=Q^+$,
$\Phi_j(A_j\cap\partial\Omega)=Q_0$, and $D\Phi_j(x)\nabla d(x)=e_n$
for every $x\in A_j\cap\Omega$. 

Let $\{\varphi_j\}$ be a partition of unity subordinated to the covering $\{A_j\}$. For every $j=1,\dots, N$ we set $u_j:=\varphi_j u$ and we consider
$$
\hat u_j(y'):= 
(D\Phi_j^{-1}(y',0))^T u_j(\Phi_j^{-1}(y',0))
$$
for a.e.\ $y'\in(-1,1)^{n-1}$.
Note that
$$
\hat u_j\cdot e_n=  u_j(\Phi_j^{-1}(y',0)) \cdot D\Phi_j^{-1}(y',0) e_n
=- u_j(\Phi_j^{-1}(y',0)) \cdot \nu_{\partial\Omega}(\Phi_j^{-1}(y',0))=0, 
$$
where we used that $\nabla d=-\nu_{\partial\Omega}\circ\pi$.
Therefore, by Step~1 there exists a function $v_j\in W^{1,1}(Q^+;\R^n)\cap L^2(Q^+;\R^n)$, with $\partial_i v_j\in L^2(Q^+;\R^n)$ for $i=1,\dots,n-1$, such that $\supp\, v_j$ is compactly contained in $A_j\cap\overline\Omega$, $v_j=\hat u_j$ on $Q_0$ and $v_j\cdot e_n=0$ a.e.\ in~$Q^+$. We define
$$
v:=\sum_{j=1}^N  (D\Phi_j)^T v_j\circ\Phi_j.
$$
Since $D\Phi_j$ is $C^{0,1}$, we have that $v\in W^{1,1}(\Omega;\R^n)\cap L^2(\Omega;\R^n)$. Moreover, by construction it is clear that
$\supp\, v\subset\Omega_a$, 
$v=u$ on $\partial\Omega$, and $v\cdot\nabla d=0$ a.e.\ in $\Omega_a$. 

It remains to check that $\Div v\in L^2(\Omega)$. 
Straightforward computations lead to
\begin{equation}\label{div-v}
\Div v=\sum_{j=1}^N  \Delta \Phi_j\cdot (v_j\circ\Phi_j)
+ \sum_{j=1}^N {\rm tr}\, \big((D\Phi_j)^T(Dv_j\circ\Phi_j)D\Phi_j\big).
\end{equation}
Let us focus on the second term on the right-hand side.
For every $j=1,\dots,n$, let $R_j\in C^1(A_k\cap\partial\Omega;\Mnn)$ be such that 
$R_j(x)\in SO(n)$ and $R_j(x)e_n=-\nu_{\partial\Omega}(x)$ for every $x\in A_j\cap\partial\Omega$.
Let $P_j:=R_j\circ\pi$.
%
%
Then,
\begin{align*}
{\rm tr}\, \big((D\Phi_j)^T(Dv_j\circ\Phi_j)D\Phi_j\big) & = \sum_{i=1}^n (D\Phi_j)^T(Dv_j\circ\Phi_j)D\Phi_j P_j e_i\cdot P_j e_i
\\
& = \sum_{i=1}^n (Dv_j\circ\Phi_j)D\Phi_j P_j e_i\cdot D\Phi_j P_j e_i.
\end{align*}
Since $D\Phi_j P_j e_n=D\Phi_j \nabla d=e_n$ and $v_j\cdot e_n=0$, the previous expression reduces to
\begin{equation}\label{expres}
{\rm tr}\, \big((D\Phi_j)^T(Dv_j\circ\Phi_j)D\Phi_j\big)
= \sum_{i=1}^{n-1} (Dv_j\circ\Phi_j)D\Phi_j P_j e_i\cdot D\Phi_j P_j e_i.
\end{equation}
For every $i=1,\dots, n-1$ and $x_0\in A_j\cap\Omega$ the vector $P_j(x_0)e_i$ is orthogonal to $\nabla d(x_0)$, hence, it is a tangent vector to the level set 
$\{x\in \Omega:\ d(x)=d(x_0)\}$; the Jacobian $D\Phi_j$ maps such vectors into vectors orthogonal to $e_n$, thus, the expression in \eqref{expres} depends only on $\partial_i v_j$ for $i=1,\dots,n-1$, and as such, it belongs to $L^2(\Omega)$. Since $v_j$ belongs to $L^2(\Omega;\R^n)$ as well, identity \eqref{div-v} implies that $\Div v\in L^2(\Omega)$.
\end{proof}

We are now in a position to prove Theorem~\ref{thm:density}.

\begin{proof}[Proof of Theorem~\ref{thm:density}.]
Upon replacing the triplet $(u,e,p)$ by $(u-w,e-Ew,p)$, we can assume that $w=0$.
The proof is subdivided into three steps.\medskip

\noindent{\it Step 1.} We first prove the statement assuming that $(u,e,p)\in\mathcal A(0, \Gamma_0)$ satisfies the additional condition $u=0$ on $\Gamma_0$ 
(hence, $p=0$ on $\Gamma_0$ by \eqref{A2}). In this step we only need $\partial\Omega$ to be of Lipschitz regularity.

The proof follows closely that of \cite[Theorem~II--3.4]{T}.
Let $k\in\N$.
Let $\{A_i\}_{i\in\N}$ be a locally finite covering of $\Omega$ and let $\{\varphi_i\}_{i\in\N}$ be a partition of unity subordinated to it.
Let $\{\varrho_\e\}$ be a family of mollifiers. For every $i$ we can find $\e_i$ such that
\begin{align}
& \|(\varphi_i u)\ast\varrho_{\e_i}-\varphi_i u\|_{L^{n/(n-1)}}\leq \frac1k\frac{1}{2^i}, \label{moll1}
\\
& \|(\nabla\varphi_i\odot u)\ast\varrho_{\e_i}-\nabla\varphi_i\odot u\|_{L^{n/(n-1)}}\leq \frac1k\frac{1}{2^i}, \label{moll3}
\\
&
\|(\varphi_i e)\ast\varrho_{\e_i}- \varphi_i e\|_{L^2}\leq \frac1k\frac{1}{2^i}, \label{moll4}
\\
& \Big| \int_\Omega |(\varphi_i p)\ast\varrho_{\e_i}|\, dx- |\varphi_i p|(\Omega)\Big|\leq \frac1k\frac{1}{2^i}. \label{moll5}
\end{align}
We then define
$$
\hat u_k:=\sum_{i=0}^\infty (\varphi_i u)\ast\varrho_{\e_i}.
$$
By \eqref{moll1}--\eqref{moll5} it is clear that $\hat u_k\in C^\infty(\Omega;\R^n)\cap LD(\Omega)$ and 
\begin{equation} \label{uk-conv}
\|\hat u_k-u\|_{L^{n/(n-1)}} \leq\frac1k.
\end{equation}
Since $\Div u={\rm tr}\, e$ in $\Omega$, \eqref{moll3} and \eqref{moll4} yield
\begin{equation}\label{div-uk}
\|\Div \hat u_k- \Div u\|_{L^{n/(n-1)}}\leq\frac1k.
\end{equation}
Moreover, one can show (see \cite[Theorem~II--3.3]{T}) that $\hat u_k=u$ on $\partial\Omega$, hence $\hat u_k=0$ on $\Gamma_0$.
Since $\Div u\in L^2(\Omega)$ and \eqref{div-uk} holds, we can construct $\psi_k\in C^\infty_c(\Omega)$ such that
\begin{equation}\label{app-div}
\|\psi_k-\Div u\|_{L^2}\leq\frac1k
\end{equation}
and
\begin{equation}\label{same-av}
\int_\Omega \psi_k(x)\, dx =\int_\Omega \Div \hat u_k(x)\, dx.
\end{equation}
We denote by $v_k$ a solution of the system
$$
\begin{cases}
\Div v_k=\psi_k-\Div \hat u_k & \text{ in }\Omega,\\
v_k=0 & \text{ on }\partial\Omega.
\end{cases}
$$
By \cite[Theorem~1]{B} (see also \cite[Theorem~2.4]{B-S}) condition \eqref{same-av} guarantee the existence of a solution $v_k\in  W^{1,n/(n-1)}_0(\Omega;\R^n)$ such that
\begin{equation}\label{vk-est}
\|v_k\|_{W^{1,n/(n-1)}}\leq C \|\psi_k-\Div \hat u_k\|_{L^{n/(n-1)}},
\end{equation}
where $C$ is a constant independent of $k$. 

We now set
\begin{align*}
& u_k:=\hat u_k+v_k, \qquad  e_k:=\sum_{i=0}^\infty (\varphi_i e)_D\ast\varrho_{\e_i} + \frac1n \psi_k \, I_{n{\times}n},
\\
& p_k:=\sum_{i=0}^\infty (\varphi_i p)\ast\varrho_{\e_i}+ \sum_{i=0}^\infty (\nabla\varphi_i \odot u)_D\ast\varrho_{\e_i} + (Ev_k)_D.
\end{align*}
It is easy to see that $(u_k,e_k,p_k)\in \mathcal A_{reg}(0,\Gamma_0)$ for every $k$. Moreover, by \eqref{uk-conv}--\eqref{app-div} and \eqref{vk-est} it is clear that $u_k$ converges
to $u$ strongly in $L^{n/(n-1)}(\Omega;\R^n)$ and by \eqref{moll4} and \eqref{app-div} that $e_k$ converges to $e$ strongly in $L^2(\Omega;\Mn)$, as $k\to\infty$. To conclude the proof, it is enough to show that
\begin{equation}\label{pk-est}
\limsup_{k\to\infty}\int_\Omega |p_k|\, dx\leq |p|(\Omega)=|p|(\Omega\cup\Gamma_0).
\end{equation}
Indeed, if \eqref{pk-est} holds, then $\{p_k\}$ is bounded in $M_b(\Omega\cup\Gamma_0;\MD)$. In particular, there exists $q\in M_b(\Omega\cup\Gamma_0;\MD)$ such that, up to subsequences, 
\begin{equation}\label{pk-conv-000}
p_k\wto q \quad \text{weakly$^*$ in } M_b(\Omega\cup\Gamma_0;\MD).
\end{equation}
Since $Eu_k=e_k+p_k$ in $\Omega$, the convergence of $\{u_k\}$ and $\{e_k\}$ imply that $Eu=e+q$ in $\Omega$, that is, $q=p$ in $\Omega$.
On the other hand, by lower semicontinuity of the norm $\|\cdot\|_{M_b}$ with respect to weak$^*$ convergence and by \eqref{pk-est}, we obtain
$$
|q|(\Omega\cup\Gamma_0)=|p|(\Omega)+|q|(\Gamma_0)
\leq \liminf_{k\to\infty}\int_\Omega |p_k|\, dx\leq  \limsup_{k\to\infty}\int_\Omega |p_k|\, dx\leq |p|(\Omega).
$$ 
Hence, $q=0$ on $\Gamma_0$, which implies that $q=p$ on $\Omega\cup\Gamma_0$. Thus, \eqref{pk-conv-000}
gives \eqref{pk1} (note that in \eqref{pk-conv-000} the whole sequence converges, since the limit is uniquely determined)
and the equality above gives \eqref{pk2}.

We now prove \eqref{pk-est}.
By \eqref{moll3} we have that
$$
\Big\| \sum_{i=0}^\infty (\nabla\varphi_i \odot u)_D\ast\varrho_{\e_i}\Big\|_{L^{n/(n-1)}}\to 0,
$$
as $k\to\infty$, and by \eqref{div-uk}, \eqref{app-div}, and \eqref{vk-est} we deduce that $(Ev_k)_D$ is converging to $0$ in $L^{n/(n-1)}(\Omega;\MD)$.
Finally, 
$$
\sum_{i=0}^\infty \int_\Omega |(\varphi_i p)\ast\varrho_{\e_i}|\, dx
\leq \sum_{i=0}^\infty \int_\Omega \varphi_i \,d |p| =|p|(\Omega).
$$
Combining these observations together, we deduce \eqref{pk-est}.
\medskip 

\noindent{\it Step 2.} We now show that any triplet $(u,e,p)\in\mathcal A(0, \Gamma_0)$ can be approximated in the sense of \eqref{uk}--\eqref{pk2}
by a sequence of triplets $(u_k,e_k,p_k)$ in $\mathcal A(0, \Gamma_0)$ such that $u_k=0$ on $\Gamma_0$ (hence, $p_k=0$ on $\Gamma_0$ by \eqref{A2}) for every $k$. In this step we will use the $C^{2,1}$ regularity of $\partial\Omega$, since the construction will be based on Theorem~\ref{AGlemma}.

Let $(u,e,p)\in\mathcal A(0, \Gamma_0)$ and let $\chi_{\Gamma_0}$ be the characteristic function of $\Gamma_0$. By Theorem~\ref{AGlemma} there exists a function $v\in W^{1,1}(\Omega;\R^n)$, with $\Div v\in L^2(\Omega)$, such that $\supp\, v\subset \Omega_a$, 
$v=-\chi_{\Gamma_0}u$ on $\partial\Omega$, and
$$
v(x)\cdot\nabla d(x)=0 
$$ 
for a.e.\ $x\in\Omega_a$.
For $k\in\N$ large enough we define $\eta_k(x):=\max\{0, 1-kd(x)\}$, $x\in\Omega$, and
\begin{align*}
& u_k:=u+\eta_k v, \qquad
e_k:=e_D+\frac{1}{n}\eta_k\Div v\, I_{n{\times}n},
\\
& p_k:=p\mres\Omega + \eta_k (Ev)_D +\nabla\eta_k\odot v.
\end{align*}
Note that $\nabla\eta_k=-k\nabla d$ a.e.\ in $\Omega_{1/k}$, so that 
$\nabla\eta_k(x)\odot v(x)\in \MD$ for a.e.\ $x\in\Omega$.
Since $\Div v\in L^2(\Omega)$, we have that $e_k\in L^2(\Omega;\Mn)$. 
Moreover, $u_k=0$ on $\Gamma_0$, since by construction $v=-u$ on $\Gamma_0$.
Thus, $(u_k, e_k, p_k)\in\mathcal A(0, \Gamma_0)$ and satisfies the required additional condition on $\Gamma_0$. 

It is easy to see that $u_k\to u$ strongly in $L^{n/(n-1)}(\Omega;\R^n)$ and
$e_k\to e$ strongly in $L^2(\Omega;\Mn)$, as $k\to\infty$. To conclude, it is enough to show that
\begin{equation}\label{pk-fin}
\limsup_{k\to\infty} |p_k|(\Omega)\leq |p|(\Omega\cup\Gamma_0).
\end{equation}
Indeed, if \eqref{pk-fin} holds, then $\{p_k\}$ is bounded in $M_b(\Omega\cup\Gamma_0;\MD)$. In particular, there exists $q\in M_b(\Omega\cup\Gamma_0;\MD)$
such that 
\begin{equation}\label{pkq}
p_k\wto q \quad \text{weakly$^*$ in } M_b(\Omega\cup\Gamma_0;\MD).
\end{equation}
Let $U$ be an open set of $\R^n$ such that $U\cap\partial\Omega=\Gamma_0$ and let us consider the extensions
$$
\tilde u_k:=\begin{cases} u_k & \text{ in }\Omega, \\ 0 & \text{ in } U\setminus\Omega,
\end{cases} \qquad
\tilde e_k:=\begin{cases} e_k & \text{ in }\Omega, \\ 0 & \text{ in } U\setminus\Omega,
\end{cases} \qquad
\tilde p_k:=\begin{cases} p_k & \text{ in }\Omega\cup\Gamma_0, \\ 0 & \text{ in } U\setminus\overline\Omega,
\end{cases}
$$
and
$$
\tilde u:=\begin{cases} u & \text{ in }\Omega, \\ 0 & \text{ in } U\setminus\Omega,
\end{cases} \qquad
\tilde e:=\begin{cases} e & \text{ in }\Omega, \\ 0 & \text{ in } U\setminus\Omega,
\end{cases} \qquad
\tilde q:=\begin{cases} q & \text{ in }\Omega\cup\Gamma_0, \\ 0 & \text{ in } U\setminus\overline\Omega.
\end{cases}
$$
Then, $\tilde u_k\to\tilde u$ strongly in $L^{n/(n-1)}(\Omega;\R^n)$, $\tilde e_k\to\tilde e$ strongly in $L^2(\Omega;\Mn)$,
and $\tilde p_k\wto\tilde q$ weakly$^*$ in $M_b(\Omega\cup\Gamma_0;\MD)$. Since $E\tilde u_k=\tilde e_k+\tilde p_k$ in $\Omega\cup U$,
we deduce that $E\tilde u=\tilde e+\tilde q$ in $\Omega\cup U$, hence $q=p$ in $\Omega\cup\Gamma_0$.
Therefore, \eqref{pkq} and \eqref{pk-fin} yield \eqref{pk1} and \eqref{pk2}.

We now prove \eqref{pk-fin}.
By definition of $p_k$ we have that
$$
|p_k|(\Omega) \leq |p|(\Omega) + \int_\Omega |\eta_k (Ev)_D|\, dx + \int_\Omega |\nabla\eta_k\odot v|\,dx.
$$
It is immediate to see that the second term on the right-handside converges to $0$, as $k\to\infty$. Thus, to prove the claim \eqref{pk-fin},
it suffices to prove that
\begin{equation}\label{pk-fin-1}
\limsup_{k\to\infty} \int_\Omega |\nabla\eta_k\odot v|\,dx \leq |p|(\Gamma_0).
\end{equation}
Using the definition of $\eta_k$ we obtain
\begin{equation}\label{goloc0}
\int_\Omega |\nabla\eta_k\odot v|\,dx
= k \int_{\Omega_{1/k}} |\nabla d\odot v|\,dx.
\end{equation}
For $k$ sufficiently small we can cover $\Omega_{1/k}$ with a finite number of open sets $A_j$, $j=1,\dots, N$, such that
for every $j=1,\dots, N$ the map $\Psi_j:(-1,1)^{n-1}\times(-\frac2k,\frac2k)\to A_j$ given by $\Psi_j(x',x_n)=(x',g_j(x'))-x_n\nu_{\partial\Omega}(x',g_j(x'))$
is a $C^{1,1}$ diffeomorphism. Here $g_j$ is a $C^{2,1}$ function whose subgraph represents $\Omega$ in $A_j$. In other words,
we may assume that $\Psi_j((-1,1)^{n-1}{\times}(0,\frac1k))=A_j\cap\Omega$.
Let $\{\varphi_j\}$ be a partition of unity subordinated to the covering $\{A_j\}$. Then
\begin{equation}\label{goloc}
k \int_{\Omega_{1/k}} |\nabla d\odot v|\,dx=
k \int_{\Omega_{1/k}} \Big|\nabla d\odot \sum_{j=1}^N\varphi_j v\Big|\,dx
\leq
k\sum_{j=1}^N \int_{\Omega_{1/k}\cap A_j} |\nabla d\odot \varphi_j v|\,dx.
\end{equation}
By a change of variable we have
\begin{multline*}
k\int_{\Omega_{1/k}\cap A_j} |\nabla d\odot \varphi_j v|\,dx
\\
= k \int_0^{1/k}\!\!\!\int_{(-1,1)^{n-1}} |\nabla d\big(\Psi_j(x',x_n)\big)\odot (\varphi_jv)\big(\Psi_j(x',x_n)\big)| \det D\Psi_j(x',x_n)\, dx'\, dx_n.
\end{multline*}
By Fubini Theorem and the definition of trace we have that there exists a set $M$ with $\LL^1(M)=0$ such that
$$
(\varphi_jv)\big(\Psi_j(x',t)\big)\to (\varphi_jv)\big(\Psi_j(x',0)\big) \qquad \text{ in } L^1((-1,1)^{n-1};\R^n),
$$
as $t\to0^+$, $t\not\in M$, where the limit $(\varphi_jv)\big(\Psi_j(x',0)\big)$ is intended in the sense of traces. Thus, we conclude that
$$
k\int_{\Omega_{1/k}\cap A_j} |\nabla d\odot \varphi_j v|\,dx
\to  \int_{(-1,1)^{n-1}} |\nu_{\partial\Omega}\big(\Psi_j(x',0)\big)\odot w_j(x', 0)| \det D\Psi_j(x',0)\, dx'.
$$
By the area formula we have
$$
\int_{(-1,1)^{n-1}} |\nu_{\partial\Omega}\big(\Psi_j(x',0)\big)\odot w_j(x', 0)| \det D\Psi_j(x',0)\, dx'
=\int_{A_j\cap\partial\Omega} |\nu_{\partial\Omega}\odot \varphi_jv| \, d\HH^{n-1}.
$$
Combining the previous equations with \eqref{goloc0} and \eqref{goloc}, we deduce that
\begin{eqnarray*}
\limsup_{k\to\infty} \int_\Omega |\nabla\eta_k\odot v|\,dx & \leq &
\sum_{j=1}^N \int_{A_j\cap\partial\Omega} |\nu_{\partial\Omega}\odot \varphi_jv| \, d\HH^{n-1}
\\
& = & \int_{\partial\Omega} \sum_{j=1}^N \varphi_j |\nu_{\partial\Omega}\odot v| \, d\HH^{n-1}
\\
& = &  \int_{\Gamma_0} |-u\odot \nu_{\partial\Omega} |\,d\HH^{n-1}
=|p|(\Gamma_0),
\end{eqnarray*}
where we used that $v=-\chi_{\Gamma_0}u$ on $\partial\Omega$.
\medskip

\noindent{\it Step 3.} To conclude, it is enough to apply a diagonal argument, together with the remark
that bounded sets of $M_b(\Omega\cup\Gamma_0;\MD)$ are metrizable with respect to weak$^*$ convergence.
\end{proof}

When the boundary condition is prescribed on the whole boundary, that is, $\Gamma_0=\partial\Omega$, 
a different construction of the approximating sequence can be performed. This new construction requires only Lipschitz regularity of the boundary and leads to more regular approximating triplets. More precisely, we have the following theorem.

\begin{theorem}\label{thm:density-dir}
Let $\Omega\subset \R^n$ be a bounded domain with a Lipschitz boundary. Assume $\Gamma_0=\partial\Omega$. 
Let $w\in W^{1,2}(\R^n;\R^n)$ and let $(u,e,p)\in\mathcal A(w, \Gamma_0)$. Then there exists a sequence of triplets $(u_k, e_k,p_k)$ in $\mathcal A_{reg}(w, \Gamma_0)$ such that
$$
(u_k-w, e_k-Ew, p_k)\in C^\infty_c(\Omega;\R^n)\times C^\infty_c(\Omega;\Mn)\times C^\infty_c(\Omega;\MD)
$$
for every $k$, and
\begin{align}
& u_k\to u \quad \text{strongly in } L^{n/(n-1)}(\Omega;\R^n), \label{uk-d}\\
& e_k\to e \quad \text{strongly in } L^2(\Omega;\Mn), \label{ek-d}\\
& p_k\wto p \quad \text{weakly$^*$ in } M_b(\overline\Omega;\MD), \label{pk1-d} 
\end{align}
and
\begin{equation}\label{pk2-d}
\int_\Omega |p_k(x)|\,dx \to |p|(\overline\Omega),
\end{equation}
as $k\to\infty$.
\end{theorem}

\begin{proof}
Upon replacing the triplet $(u,e,p)$ by $(u-w,e-Ew,p)$, we can assume that $w=0$.

We consider the extensions
$$
\tilde u:=\begin{cases} u & \text{ in }\Omega, \\ 0 & \text{ in } U\setminus\Omega,
\end{cases} \qquad
\tilde e:=\begin{cases} e & \text{ in }\Omega, \\ 0 & \text{ in } U\setminus\Omega,
\end{cases} \qquad
\tilde p:=\begin{cases} p & \text{ in }\overline\Omega, \\ 0 & \text{ in } U\setminus\overline\Omega,
\end{cases}
$$
where $U$ is an open and bounded set such that $\Omega$ is compactly contained in $U$.
Note that $\tilde u\in BD(U)$, $\tilde e \in L^2(U;\Mn)$, $\tilde p\in M_b(U;\MD)$, and $E\tilde u=\tilde e+\tilde p$ in $U$.
In particular, we have that $\Div \tilde u\in L^2(U)$.

Since $\Omega$ is bounded and has a Lipschitz boundary, there exists a finite open cover $\{A_j\}$, $j=1,\dots,N$, of $\partial\Omega$,
made of open cubes centred at points on $\partial\Omega$, with a face orthogonal to some vector $\xi_j\in \mathbb S^1$ and such that
the set $A_j\cap\Omega$ is a Lipschitz subgraph in the direction $\xi_j$. We set $A_0:=\Omega$ and $\xi_0:=0$.
For every $j=0,\dots,N$ and every $k\in\N$ we introduce the translation
$$
\tau_{j,k}(x):=x+\frac1k \xi_j \qquad \text{ for } x\in \R^n.
$$
Finally, let $\{\varphi_j\}$ be a partition of unity subordinated to $\{A_j\}$ and let $\{\varrho_\e\}$ be a family of mollifiers.

We define
$$
\begin{array}{c}
\displaystyle
\tilde u_k:=\sum_{j=0}^N (\varphi_j \tilde u)\circ\tau_{j,k}, \qquad
\tilde e_k:= \sum_{j=0}^N (\varphi_j \tilde e)_D\circ\tau_{j,k} +\frac1n \Div\tilde u_k\, I_{n\times n},
\smallskip\\
\displaystyle
\tilde p_k:=\sum_{j=0}^N \tau_{j,k}^\#(\varphi_j \tilde p) + \sum_{j=0}^N (\nabla\varphi_j \odot \tilde u)_D\circ\tau_{j,k},
\end{array}
$$
where $\tau_{j,k}^\#(\varphi_j \tilde p)$ denotes the pull-back measure of $\varphi_j \tilde p$.
We observe that $\tilde u_k\in BD(\Omega)$, $(\tilde e_k)_D\in L^2(\Omega;\MD)$,
$\Div\tilde u_k={\rm tr}\,\tilde e_k\in L^{n/(n-1)}(\Omega)$,
$\tilde p_k\in M_b(\Omega;\MD)$, and they all have compact support in $\Omega$
for $k$ sufficiently small. 
Moreover, $E\tilde u_k =\tilde e_k+\tilde p_k$ in $\Omega$
and, as $k\to\infty$, we have 
\begin{eqnarray}
& \tilde u_k\to u \qquad \text{ strongly in } L^{n/(n-1)}(\Omega;\R^n), \label{Dir-huk}
\\
& \Div \tilde u_k \to \Div u \qquad \text{ strongly in } L^{n/(n-1)}(\Omega), \label{dive0}
\\
& (\tilde e_k)_D \to e_D \qquad \text{ strongly in } L^2(\Omega;\MD), \label{Dir-hek}
\end{eqnarray}
and
\begin{equation}\label{resto}
\sum_{j=0}^N (\nabla\varphi_i \odot \tilde u)_D\circ\tau_{j,k} \to 0 \qquad  \text{ strongly in } L^{n/(n-1)}(\Omega;\MD).
\end{equation}
Since $\tau_{j,k}$ is an outward translation on $A_j\cap\Omega$, we have that
$$
\sum_{j=0}^N | \tau_{j,k}^\#(\varphi_j \tilde p)|(\Omega)
\leq \sum_{j=0}^N |\varphi_j \tilde p|(U)
= \sum_{j=0}^N \int_U \varphi_j \, d|\tilde p|
= |p|(\overline\Omega)
$$
for every $k\in\N$, hence by \eqref{resto} we deduce that
\begin{equation}\label{dir-pk-est-2}
\limsup_{k\to\infty} |\tilde p_k|(\Omega)\leq |p|(\overline\Omega).
\end{equation}

We now set
$$
\hat u_k:=\tilde u_k\ast \varrho_{\e_k}, \qquad \hat e_k:=\tilde e_k\ast \varrho_{\e_k}, \qquad \hat p_k:=\tilde p_k\ast \varrho_{\e_k}, 
$$
where $\e_k$ is chosen so that $\hat u_k\in C^\infty_c(\Omega;\R^n)$, $\hat e_k\in C^\infty_c(\Omega;\Mn)$, $\hat p_k\in C^\infty_c(\Omega;\MD)$,
and
\begin{align}
& \|\hat u_k -\tilde u_k\|_{L^{n/(n-1)}}\leq\frac1k, \label{hatti-uk}
\\
& \|\Div\hat u_k -\Div\tilde u_k\|_{L^{n/(n-1)}}\leq\frac1k, \label{dive1}
\\
& \|(\hat e_k)_D -(\tilde e_k)_D\|_{L^2} \leq\frac1k, \label{hatti-ek}
\\
& \Big| \int_\Omega|\hat p_k|\, dx- |\tilde p_k|(\Omega)\Big|\leq\frac1k, \label{hatti-pk}
\end{align}
where the last inequality follows from the fact that $|\tilde p_k|(\partial\Omega)=0$, hence mollifications of $\tilde p_k$ strictly converge to $\tilde p_k$ in $\Omega$. Clearly, we still have that $E\hat u_k=\hat e_k+\hat p_k$ in $\Omega$.

We now introduce a correction for $\Div \hat u_k$.
Since $\Div u\in L^2(\Omega)$ and $\Div\hat u_k\to\Div u$ in $L^{n/(n-1)}(\Omega)$
by \eqref{dive0} and \eqref{dive1}, we can construct $\psi_k\in C^\infty_c(\Omega)$ such that
\begin{equation}\label{app-div-dir}
\|\psi_k-\Div u\|_{L^2}\leq\frac1k,
\end{equation}
and
\begin{equation}\label{same-av-dir}
\int_\Omega \psi_k(x)\, dx =\int_\Omega \Div \hat u_k(x)\, dx.
\end{equation}
We denote by $v_k$ a solution of the system
$$
\begin{cases}
\Div v_k=\psi_k-\Div \hat u_k & \text{ in }\Omega,\\
v_k=0 & \text{ on }\partial\Omega.
\end{cases}
$$
By \cite[Remark~4]{B} (see also \cite[Theorem~2.4]{B-S}) condition \eqref{same-av-dir} guarantees the existence of a solution $v_k\in C^\infty_c(\Omega;\R^n)$ such that
$$
\|v_k\|_{W^{1,n/(n-1)}}\leq C \|\psi_k-\Div \hat u_k\|_{L^{n/(n-1)}},
$$
where $C$ is a constant independent of $k$. Combining this inequality with \eqref{dive0}, \eqref{dive1}, and \eqref{app-div-dir},
we deduce that
\begin{equation}\label{Dir-vk-est}
v_k\to 0 \qquad \text{ strongly in } W^{1,n/(n-1)}(\Omega;\R^n).
\end{equation}

We are now ready to define the approximating sequence. We set
$$
u_k:=\hat u_k+v_k, \qquad  e_k:= \hat e_k + \frac1n \Div v_k \, I_{n{\times}n},
\qquad
p_k:=\hat p_k+ (Ev_k)_D.
$$
It is immediate to see that $u_k\in C^\infty_c(\Omega;\R^n)$, $e_k\in C^\infty_c(\Omega;\Mn)$, and $p_k\in C^\infty_c(\Omega;\MD)$.
Moreover, $Eu_k=e_k+p_k$ in $\Omega$, so that $(u_k,e_k,p_k)\in \mathcal A_{reg}(0, \Gamma_0)$ for every $k$. 
By \eqref{Dir-huk}, \eqref{hatti-uk}, and \eqref{Dir-vk-est} we immediately deduce \eqref{uk-d}.
Note that
$$
e_k=(\hat e_k)_D+ \frac1n \psi_k\, I_{n\times n},
$$
hence, \eqref{Dir-hek}, \eqref{hatti-ek}, and \eqref{app-div-dir} yield \eqref{ek-d}. 
Finally, by \eqref{dir-pk-est-2}, \eqref{hatti-pk}, and \eqref{Dir-vk-est} we deduce that
\begin{equation}\label{dir-pk-est}
\limsup_{k\to\infty} \int_\Omega |p_k|\, dx \leq \limsup_{k\to\infty} \int_\Omega |\hat p_k|\, dx
 \leq \limsup_{k\to\infty} |\tilde p_k|(\Omega) \leq 
| p|(\overline\Omega).
\end{equation}
This last inequality is enough to conclude.
Indeed, extending to $0$ the triplets $(u_k,e_k,p_k)$ outside $\overline\Omega$,
we have by \eqref{dir-pk-est} that $\{p_k\}$ is bounded in $M_b(U;\MD)$, hence, up to subsequences,
\begin{equation}\label{www}
p_k\wto q \qquad \text{ weakly$^*$ in } M_b(U;\MD).
\end{equation}
By \eqref{uk-d} and \eqref{ek-d} we deduce that $E\tilde u=\tilde e+q$ in $U$, hence, in particular, $q=p$ in $\overline\Omega$.
This fact, together with \eqref{dir-pk-est} and \eqref{www}, yields \eqref{pk1-d} and \eqref{pk2-d}.
This concludes the proof of the theorem.
\end{proof}

\section{The Relaxation Result}

In this section we apply the density theorems of Section~\ref{sec:density} to characterise the relaxation of the Hencky model.
In the notation of the introduction we prove that \eqref{min pb 2} is the relaxed problem of \eqref{min pb} when $p_{i-1}=0$.

For the sake of notation it is convenient to express the involved functionals in terms of the displacement $u$ 
and of the elastic strain $e$, only. The plastic strain $p$ can be always recovered a posteriori 
by the kinematic compatibility condition. 

\begin{definition}
Let $w\in W^{1,2}(\R^n;\R^n)$ and let $\Gamma_0$ be an open subset of $\partial\Omega$. 
We define $\mathcal B_{reg}(w, \Gamma_0)$ as the class of all pairs $(u,e)\in LD(\Omega)\times L^2(\Omega;\Mn)$
such that 
\begin{align}
& \Div u={\rm tr}\,e \quad \text{a.e.\ in }  \Omega, \label{Breg1}
\\
& u=w \quad \text{ on }\Gamma_0. \label{Breg2}
\end{align}
\end{definition}

\begin{definition}
We define $\mathcal B$ as the class of all pairs $(u,e)\in BD(\Omega)\times L^2(\Omega;\Mn)$
such that \eqref{Breg1} holds.
\end{definition}

Note that $(u,e)\in \mathcal B_{reg}(w, \Gamma_0)$ if and only if there exists $p\in L^1(\Omega;\MD)$
such that $(u,e,p)\in \mathcal A_{reg}(w, \Gamma_0)$.
Given any $w\in W^{1,2}(\R^n;\R^n)$ and any $\Gamma_0$ open subset of $\partial\Omega$, we have that
$(u,e)\in \mathcal B$ if and only if there exists $p\in M_b(\Omega\cup\Gamma_0;\MD)$
such that $(u,e,p)\in \mathcal A(w, \Gamma_0)$.

We are now in a position to state the main result of this section.

\begin{theorem}\label{thm:relax}
Let $Q:\Mn\to[0,+\infty)$ be a positive definite quadratic form and let $H:\MD\to[0,+\infty]$ a convex and positively one-homogeneous function
such that
\begin{equation}\label{H-coer}
r|\xi|\leq H(\xi)\leq R|\xi| \qquad \text{ for every }\xi\in\MD,
\end{equation}
with $0<r\leq R$.
Assume one of the two following conditions: either\smallskip
\begin{itemize}
\item[(i)] $\Omega\subset \R^n$ is an open and bounded set with a $C^{2,1}$ boundary and $\Gamma_0$ an open subset of~$\partial\Omega$,
\end{itemize}
or
\begin{itemize}
\item[(ii)] $\Omega\subset \R^n$ is a bounded domain with a $Lipschitz$ boundary and $\Gamma_0=\partial\Omega$.
\end{itemize}
\smallskip
Let $w\in W^{1,2}(\R^n;\R^n)$. Let $\mathcal F: BD(\Omega)\times L^2(\Omega;\Mn)\to[0,+\infty]$
given by
$$
\mathcal F(u,e):=
\begin{cases}
\displaystyle
 \int_\Omega Q(e(x))\, dx +\int_\Omega H(Eu(x)-e(x))\, dx
& \text{ if } (u,e)\in {\mathcal B}_{reg}(w,\Gamma_0),
\medskip
\\
+\infty
& \text{ otherwise.} 
\end{cases}
$$
The lower semicontinuous envelope of $\mathcal F$, with respect to the product of the $L^1(\Omega;\R^n)$-strong topology and the $L^2(\Omega;\Mn)$-weak topology, is the functional $\mathcal G: BD(\Omega)\times L^2(\Omega;\Mn)\to[0,+\infty]$
given by
$$
\mathcal G(u,e):=
\int_\Omega Q(e(x))\, dx +\HH_\Omega(Eu-e) +\int_{\Gamma_0} H((w-u)\odot\nu_{\partial\Omega})\, d\HH^{n-1}
$$
if $(u,e)\in {\mathcal B}$, and
$\mathcal G(u,e):=+\infty$ otherwise. 
\end{theorem}

\begin{proof}
Since the functional $\mathcal F$ is coercive in the elastic strain $e$ with respect to the weak topology of $L^2(\Omega;\Mn)$,
the lower semicontinuous envelope $\overline{\mathcal F}$ can be characterised sequentially as
\begin{multline*}
\overline{\mathcal F}(u,e)=\inf\Big\{ \liminf_{j\to\infty} \mathcal F(u_j, e_j): \ u_j\to u \text{ strongly in } L^1(\Omega;\R^n), 
\\
e_j\wto e \text{ weakly in } L^2(\Omega;\Mn) \Big\}.
\end{multline*}

We first prove that $\mathcal G\leq\overline{\mathcal F}$. Let $(u,e)\in BD(\Omega)\times L^2(\Omega;\Mn)$ and let $u_j\to u$ 
strongly in $L^1(\Omega;\R^n)$ and $e_j\wto e$ weakly in $L^2(\Omega;\Mn)$. We want to show that
$$
\liminf_{j\to\infty} {\mathcal F}(u_j, e_j)\geq \mathcal G(u,e).
$$
Without loss of generality we can assume that 
$$
\liminf_{j\to\infty} {\mathcal F}(u_j, e_j)<+\infty
$$
and, up to subsequences, that the above liminf is a limit. Thus, we deduce by \eqref{H-coer} that there exists a constant $C>0$ such that
$$
\int_\Omega |Eu_j(x)-e_j(x)|\, dx \leq C
$$
for every $j$. Up to extracting a further subsequence, we can thus assume that $p_j:=Eu_j-e_j\wto p$ weakly$^*$ in $M_b(\Omega;\MD)$.
Since $p=Eu-e$ in $\Omega$, we have $(u,e)\in\mathcal B$.

Since $\Gamma_0$ is an open subset of $\partial\Omega$, there exists an open set $U$ in $\R^n$ such that $\Gamma_0=\partial\Omega\cap U$.
We define
$$
\tilde u_j:=\begin{cases}
u_j & \text{ in }\Omega,
\\
w & \text{ in } U\setminus\Omega,
\end{cases}
\qquad
\tilde e_j:=\begin{cases}
e_j & \text{ in }\Omega,
\\
Ew & \text{ in } U\setminus\Omega,
\end{cases}
$$
and
$$
\tilde u:=\begin{cases}
u & \text{ in }\Omega,
\\
w & \text{ in } U\setminus\Omega,
\end{cases}
\qquad
\tilde e:=\begin{cases}
e & \text{ in }\Omega,
\\
Ew & \text{ in } U\setminus\Omega.
\end{cases}
$$
Clearly $\tilde u_j\wto \tilde u$ weakly$^*$ in $BD(U)$ and $\tilde e_j\wto\tilde e$ weakly in $L^2(U;\Mn)$.
By definition of the extensions and by lower semicontinuity we deduce that
\begin{eqnarray*}
\liminf_{j\to\infty} \mathcal F(u_j,e_j) & = &  \liminf_{j\to\infty} \int_\Omega Q(e_j(x))\, dx + \HH_U(E\tilde u_j-\tilde e _j)
\\
& \geq & \int_\Omega Q(e(x))\, dx + \HH_U(E\tilde u-\tilde e)
= \mathcal G(u,e).
\end{eqnarray*}

We now prove that $\mathcal G\geq\overline{\mathcal F}$. Let $(u,e)\in \mathcal B$. Set
$$
p:= Eu-e \text{ in }\Omega, \qquad p:=(w-u)\odot\nu_{\partial\Omega}\HH^{n-1} \text{ on } \Gamma_0.
$$
By Theorem~\ref{thm:density} or Theorem~\ref{thm:density-dir} (according to the validity of assumption (i) or (ii), respectively) 
there exists a sequence $(u_j,e_j, p_j)\in {\mathcal A}_{reg}(w,\Gamma_0)$ such that
\begin{align}
& u_j\to u \quad \text{strongly in } L^{n/(n-1)}(\Omega;\R^n), \\
& e_j\to e \quad \text{strongly in } L^2(\Omega;\Mn), \label{ek2-e} \\
& p_j\to p \quad \text{weakly$^*$ in } M_b(\Omega\cup\Gamma_0;\MD), \label{pk2-e}
\end{align}
and
\begin{equation}\label{pk2-2e}
\int_\Omega |p_j(x)|\,dx \to |p|(\Omega\cup\Gamma_0). 
\end{equation}
We now extend $p_j$ and $p$ by $0$ to an open subset $U$ such that $U\cap\partial\Omega=\Gamma_0$, and
we call $\tilde p_j$ and $\tilde p$ the extensions, respectively. By \eqref{pk2-e} and \eqref{pk2-2e} we have that
$\tilde p_j\to \tilde p$ strictly in $M_b(U;\MD)$, hence the Reshetnyak Theorem (see \cite[Theorem~2.39]{AFP})
implies that
\begin{equation}\label{resh}
\int_\Omega H(p_j(x))\, dx=\HH_U(\tilde p_j) \to \HH_{U}(\tilde p).
\end{equation}
Note that
$$
\HH_{U}(\tilde p)=\HH_\Omega(Eu-e) +\int_{\Gamma_0} H((w-u)\odot\nu_{\partial\Omega})\, d\HH^{n-1}.
$$
Since $(u_j,e_j)\in \mathcal B_{reg}(w,\Gamma_0)$ and $p_j=Eu_j-e_j$,
convergences \eqref{ek2-e} and \eqref{resh} imply that
$$
\lim_{j\to\infty} \mathcal F(u_j,e_j)= \mathcal G(e,u).
$$
This concludes the proof.
\end{proof}

\bigskip\bigskip

\noindent
\textbf{Acknowledgements.}
The author acknowledges support by GNAMPA--INdAM and by the European Research Council under Grant No.\ 290888
``Quasistatic and Dynamic Evolution Problems in Plasticity and Fracture''.

\end{document}